\documentclass[12.5pt]{article}

\usepackage{graphicx}
\usepackage{amssymb}
\usepackage{exercise}
\usepackage{amsmath}
\usepackage{amsthm}
\usepackage[all,knot]{xy}
\usepackage{tikz}
\usetikzlibrary{matrix}
\usepackage{asymptote}
\usepackage{hyperref}
\usepackage{braket}

\newtheorem{mydef}{Definition}
\newtheorem{lemma}{Lemma}
\newtheorem{mythm}{Theorem}
\newtheorem{mypro}{Proposition}
\newtheorem{mycon}{Conjecture}

\begin{document}
\large

\title{Quantum groups obtained from solutions to the parametrized Yang-Baxter equation}
\author{Valentin Buciumas}
\date{}
\maketitle

\thispagestyle{empty}

\tableofcontents

\renewcommand{\abstractname}{}

\begin{abstract}
\hspace{-0.59cm} $\bold{Abstract}$: In this article we use a parametrized version of the FRT construction to construct two new coquasitriangular Hopf algebras. The first one, $\widehat{SL_q(2)}$, is a quantization of the coordinate ring on affine $SL(2)$. We show that there is a duality relation between this object and the more well-known $U_q(\widehat{\mathfrak{sl}_2})$. We then build certain irreducible  comodules of this Hopf algebra and prove an irreducibility criterion for their tensor product in the spirit of Chari and Pressley. 

The second object is built from a solution of the parametrized Yang-Baxter equation with parameter group GL$(2, \mathbb{C}) \times$GL$(1,\mathbb{C})$. This solution doesn't come from any known quantum group, though it is related to both solutions coming from $U_{\pm i}(\widehat{\mathfrak{sl}_2})$ and $U_q(\widehat{\mathfrak{gl}(1|1)})$. We then study certain irreducible comodules of this newly built object.
\end{abstract}

\section{Introduction}

The idea of a quantum group, or a quasitriangular Hopf algebra, was introduced by Drinfel'd $\cite{Drinfeld}$ and Jimbo $\cite{Jimbo2}$ independently while looking for solutions to the Yang-Baxter equation. An example of quantum group is the quantized enveloping algebra of a finite or affine Lie algebra $\mathfrak{g}$ which we denote by $U_q(\mathfrak{g})$, where $q$ is a generic parameter. $U_q(\mathfrak{g})$ has the structure of a Hopf algebra and also a universal $R$-matrix, namely an invertible element $\mathcal{R}$ that lives in the completion $U_q(\mathfrak{g}) \widehat{\otimes} U_q(\mathfrak{g})$. The $R$-matrix satisfies some interesting properties that make the category of finite dimensional $U_q(\mathfrak{g})$ modules into a braided category. One of the properties of the $R$-matrix is that it satisfies the Yang-Baxter equation (YBE): $\mathcal{R}_{12} \mathcal{R}_{13} \mathcal{R}_{23} = \mathcal{R}_{23} \mathcal{R}_{13} \mathcal{R}_{12}$ where $\mathcal{R}_{12} = \mathcal{R} \otimes 1$, $\mathcal{R}_{23} = 1 \otimes \mathcal{R}$ and $\mathcal{R}_{13} = (\text{id} \otimes \tau)(\mathcal{R} \otimes 1)$ all live in the completion of $U_q(\mathfrak{g})^{\otimes 3}$.

If $\mathfrak{g}$ is a finite simple Lie algebra, then for every $V_1$, $V_2$ finite dimensional representations  of $U_q(\mathfrak{g})$, $\mathcal{R}$ will give rise to a matrix $R \in \text{End} (V_1 \otimes V_2)$ that will satisfy the YBE for matrices, namely:

\begin{equation*}
R_{12} R_{13} R_{23} = R_{23} R_{13} R_{12}
\end{equation*}
seen as an identity in $\text{End}(V_1 \otimes V_2 \otimes V_3)$. For example, if we work with $U_q(\mathfrak{sl}_2)$, and $V_1 = V_2 =V$ is the standard two dimensional representation, then the matrix $R$ will have the following formula (we write $R_q$ to highlight the dependency on $q$):

\begin{equation*}
R_q=\left ( \begin{array}{cccc}
q & 0 & 0 & 0 \\
0 & 1 & 0 & 0 \\
0 & q-q^{-1} & 1 & 0 \\
0 & 0 & 0 & q \end{array} \right )
\end{equation*}

If $\widehat{\mathfrak{g}}$ is an untwisted affine Lie algebra, $U_q(\widehat{\mathfrak{g}})$ is again a quasitriangular Hopf algebra. $\mathcal{R}$ will now give rise to solutions of the parametrized YBE with parameter group $G = \mathbb{C}^*$, namely matrices $R(x)$ for all $x \in G$ that satisfy the identity 
\begin{equation*}
R_{12}(x) R_{13}(x y) R_{23}(y) = R_{23}(y) R_{13}(xy) R_{12}(x)
\end{equation*}
for any $x, y \in G$ . For example, in the $\widehat{\mathfrak{sl}_2}$ case, Jimbo discovered the existence of a quantum evaluation operator $ev_a : U_q(\widehat{\mathfrak{sl}_2}) \to U_q(\mathfrak{sl}_2)$ for all $a \in \mathbb{C}^*$. The pullback by $ev_a$ of any representation $V$ of $U_q(\mathfrak{sl}_2)$ will give rise to a finite dimensional representation $V_a$, of the same dimension as $V$, of $U_q(\widehat{\mathfrak{sl}_2})$. If $V$ is the standard representation of $U_q(\mathfrak{sl}_2)$, we get a series of representations $V_a$ for all $a \in \mathbb{C}^*$. $\mathcal{R}$ will act on the tensor product $V_a \otimes V_b$ as follows:

\begin{equation} \label{eq:affine_YBE_solution}
R_q(x)=\left ( \begin{array}{cccc}
q-xq^{-1} & 0 & 0 & 0 \\
0 & 1-x & x(q-q^{-1}) & 0 \\
0 & q-q^{-1} & 1-x & 0 \\
0 & 0 & 0 & q-xq^{-1} \end{array} \right ), \qquad x=\frac{a}{b}
\end{equation}

Solutions to the parametrized YBE were instrumental in understanding the theory of certain lattice models. They were used to compute partition functions of such systems. The partition function allows one to understand the global behavior of the system by looking at its local properties. The 6-vertex model is one such example. Each state of the system is modeled by labeling the edges of a finite two dimensional rectangular lattice by $\pm$ signs. Each vertex will then be assigned a Boltzmann weight which depends on the labeling of the edges connected to the vertex. The product of all the Boltzmann weights of vertices in a given state will produce the Boltzmann weight of the state, while summing over all the Boltzmann weights of possible states of the system will result in the partition function. The partition function, the object that best describes the system, is the thing physicists are really interested in.

Transfer matrices encode information about rows in such a model. Baxter $\cite{Baxter}$ showed that solutions of the parametrized YBE are needed in order to prove that transfer matrices commute. This allowed him to compute the partition function of the six vertex model. In the field-free case of the 6-vertex model, one uses ``almost'' the solution $R_q(x)$ corresponding to the standard finite dimensional evaluation representation of $U_q(\widehat{\mathfrak{sl}_2})$. However, the relation between $U_q (\widehat{\mathfrak{sl}_2})$ and the 6-vertex model is deeper than this. For example, it was showed that the one point function for the 6-vertex model can be expressed as the quotient of the string function by the character of the basic representation of $U_q(\widehat{\mathfrak{sl}_2})$ (see $\cite{HongKang}$ for more details).

Even though the motivation for constructing quantum groups was to find solutions of the YBE, one can ask the following question: starting with just a solution of the Yang-Baxter equation, can you build a quantum group out of it? For example Jimbo $\cite{Jimbo1}$ wrote down the solutions to the parametrized YBE corresponding to quantum affine algebras before the universal $R$-matrix was constructed. The answer is close to yes. It is based on the Faddeev-Reshetikhin-Takhtajan (FRT) construction $\cite{Faddeev}$ which creates a coquasitriangular bialgebra, an object which is in duality with a quasitriangular bialgebra, also known as a quantum group.  

The FRT construction can be understood in terms of the reconstruction theorem for braided categories. The most basic reconstruction theorem, also known as a Tannakian theorem for bialgebras, was introduced by Saavedra-Rivano in $\cite{Saavedra-Rivano}$ and takes the following form. Let $k$ be a field, and let $\mathcal{C}$ be a monoidal category which is abelian and essentially small. If $\omega : \mathcal{C} \to  \text{Vect}_k$ is a monoidal functor which is exact and faithful, then there exists a coalgebra $A$ such that $\omega$ factors through an equivalence of categories $\mathcal{C} \to \text{Comod}_A$ between $\mathcal{C}$ and the category of $A$ comodules. Using the monoidal structure on $\mathcal{C}$, it was shown that $A$ is a bialgebra. For this construction, Ulbrich $\cite{Ulbrich}$ showed that if $\mathcal{C}$ is rigid, then $A$ will be a Hopf algebra. Majid $\cite{Majid}$ then proved that if $\mathcal{C}$ is a braided, not necessarily rigid category, then A becomes a coquasitriangular bialgebra, while Pfeiffer $\cite{Pfeiffer}$ proved a similar theorem for modular categories. 

We now briefly explain the FRT construction. Let $V$ be a vector space and $R \in \text{End}(V \otimes V)$ an invertible solution of the Yang-Baxter equation. FRT construct a bialgebra $A_R$ such that $V$ is an $A_R$ comodule and $\tau R : V \otimes V \to V \otimes V$ is an $A_R$ homomorphism. Their construction can be understood as follows: $A_R$ is the coalgebra obtained by using the reconstruction theorem for the braided monoidal category $\mathcal{C}$ generated by $V$ whose braiding map is given by $\tau R : V \otimes V \to V \otimes V$. The braiding ensures that the bialgebra $A_R$ is coquasitriangular. If one slightly modifies the category, then $A_R$ will become a coquasitriangular Hopf algebra. If we start with $R_q$ to be the $U_q(\mathfrak{sl}_2)$ solution to the YBE in the standard representation, we obtain $SL_q(2)$, a quantization of the coordinate ring of $SL(2, \mathbb{C})$. There is a duality relation between $U_q(\mathfrak{sl}_2)$ and $SL_q(2)$ which is to be expected.

We now present the main results in this paper. In the first part we use a parametrized version of the FRT construction with the solution $R_q(x)$ corresponding to the $R$-matrix of the quantum group $U_q(\widehat{\mathfrak{sl}_2})$ and construct a new quantum group $\widehat{SL_q(2)}$. We introduce an affine version of the quantum determinant which allows us to define an antipode, showing that $\widehat{SL_q(2)}$ is a Hopf algebra. We then prove there is a duality relation between $\widehat{SL_q(2)}$ and $U_q(\widehat{\mathfrak{sl}_2})$. We show that $\widehat{SL_q(2)}$ has a set of irreducible finite dimensional comodules that are related to the evaluation modules of $U_q(\widehat{\mathfrak{sl}_2})$ via the duality relation and satisfy similar properties to the evaluation modules discovered by Chari and Pressley $\cite{ChariPressley1}$ . Finally we discuss the construction of $\widehat{SL_q(n)}$ and what happens in other types.  

In the second part we build a quantum group from a solution of the parametrized YBE with non-commutative parameter group. Korepin $\cite{Korepin}$ and Bump, Brubaker and Friedberg $\cite{Bump2}$ independently discovered a solution to the parametrized YBE with non-commutative parameter group $\Gamma := SL(2, \mathbb{C}) \times GL(1, \mathbb{C})$ that does not correspond to any known quantum group. This solution is related to the six-vertex model, it is an expansion at $q=\pm i$ of the solution $R_q(x)$ defined in equation $\ref{eq:affine_YBE_solution}$. It is also an expansion of the Perk-Schultz solution of the YBE $\cite{PerkSchultz}$ which can be obtained from the $R$-matrix of the quantum super group $U_q(\widehat{\mathfrak{gl}(1|1)})$ in the standard representation $\cite{Kojima}$. It should be of interest to physicists since it is the center of the disordered regime of the six-vertex model and is contained in the free fermionic eight-vertex model of Fan and Wu $\cite{FanWu1}$, $\cite{FanWu2}$.

We use the reconstruction theorem to associate a coquasitriangular bialgebra $\mathcal{A}_{ff}$ to this solution of the parametrized YBE that has standard two dimensional comodules $V_x$ for all $x \in \Gamma$. We find a new set of two dimensional corepresentations. We give conditions for when the tensor product of finitely many standard comodules is irreducible and classify the subcomodules of $V_x \otimes V_y$. Finally, we give a conjecture regarding the dimension of any finite dimensional comodule and we talk about a dual construction.   

$\bold{Acknowledgements}$. I would like to thank my supervisor Daniel Bump for continued guidance and support. This work was partly supported by the NSF grant DMS-1001099. 

\section{Preliminary notions}

\subsection{Quasitriangular Hopf algebras}

In this subsection we give basic definitions from the theory of quantum groups. Most of these definitions can be found in standard texts, for example $\cite{ChariPressley}$. 

All vector spaces will be over a field $k$ of characteristic $0$. $I$ will denote the identity matrix, $I \in $End$(V)$, and $\tau$ will denote the flip, $\tau (v_i \otimes v_j) = v_j \otimes v_i$.

Given a vector space $V$, we say that $R \in \textnormal{End}(V \otimes V)$ is a solution to the parametrized YBE if the following equation holds:
\begin{equation}
R_{12} R _{13} R_{23}  = R_{23} R_{13}  R_{12}
\end{equation}
seen as an identity in End($V \otimes V \otimes V$), where $R_{12} = R \otimes I$, $R_{23} = I \otimes R$ and $R_{13} = (I \otimes\tau) (R \otimes I) (I \otimes \tau )$.

Given a group $\Gamma$ and a vector space $V$, we say that $R: \Gamma \to \textnormal{End}(V \otimes V)$ is a solution to the parametrized YBE if the following equation holds for all $\alpha, \beta \in \Gamma$:
\begin{equation}
R_{12}(\alpha) R _{13}(\alpha \cdot \beta) R_{23} (\beta) = R_{23}(\beta) R_{13} (\alpha \cdot \beta) R_{12}(\alpha)
\end{equation}

\begin{mydef}
A quasitriangular Hopf algebra $H$ is a Hopf algebra with an invertible element $\mathcal{R} \in H \otimes H$ that satisfies the following relations for all $h \in H$:
\begin{itemize}
\item $\Delta^{op}(h) = \mathcal{R}   \Delta(h)    \mathcal{R}^{-1}$ 
\item $(\Delta \otimes 1) (\mathcal{R}) = \mathcal{R}_{13} \mathcal{R}_{23}$
\item $(1 \otimes \Delta) (\mathcal{R}) = \mathcal{R}_{13} \mathcal{R}_{12}$
\end{itemize}
where $ \mathcal{R}_{12} = \mathcal{R}\otimes 1$, etc.
\end{mydef}

Given a quasitriangular Hopf algebra $H$ with a module $V$, if $R$ is the action of $\mathcal{R}$ on $V \otimes V$, then $R$ will satisfy the YBE.

The notion of a coquasitriangular Hopf algebra was introduced by Majid in $\cite{Majid}$. It is dual to the notion of a quasitriangulr Hopf algebra. 

\begin{mydef} \label{def:dualqtha}
A coquasitriangular Hopf algebra is a Hopf algebra A with a linear map $\mathcal{R} :A \otimes A \to k$ such that for every $a, b, c \in A$:
\begin{equation} 
\begin{split} 
a_{(1)} b_{(1)} \mathcal{R}(b_{(2)} \otimes a_{(2)}) =  \mathcal{R} (b_{(1)} \otimes a_{(1)}) b_{(2)} a_{(2)} \\
\mathcal{R} (ab \otimes c) = \mathcal{R} (a \otimes c_{(1)}) \mathcal{R}(b \otimes c_{(2)})   \\
\mathcal{R}(a \otimes bc) = \mathcal{R}(a_{(1)}\otimes b) \mathcal{R}(a_{(2)}\otimes c)
\end{split}
\end{equation}
$\mathcal{R}$ also has to be convolution-invertible, which means that there is $\mathcal{R}^{-1}: A \otimes A \to k$ such that $\mathcal{R}(a_{(1)} \otimes b_{(1)}) \mathcal{R}^{-1}(a_{(2)} \otimes b_{(2)}) = \epsilon(ab)$.
\end{mydef}

The category of $A$ comodules becomes braided if we set $\Psi _{V_1, V_2} = (\mathcal{R} \otimes \text{id}) (\text{id} \otimes \tau \otimes \text{id}) (\alpha_1 \otimes \alpha_2) \tau: V_1 \otimes V_2 \to V_2 \otimes V_1$ where $\alpha_1$ and $\alpha_2$ are the coaction maps for the comodules $V_1$ and $V_2$.

\begin{mydef} \label{def:duality}
A duality relation relation between two Hopf algebras $H$ and $A$ is a linear map $\braket{,}: H \otimes A \to k$ that satisfies
\begin{itemize}
\item $\braket{uv, x} = \braket{u, x_{(1)}} \braket{v, x_{(2)}}$,
\item $\braket{u, xy} = \braket{u_{(1)}, x} \braket{u_{(2)}, y}$,
\item $\braket{u, 1} = \epsilon(u)$,
\item $\braket{1, x} = \epsilon(x)$,
\item $\braket{S(u), x} = \braket{u, S(x)}$.
\end{itemize}
for all $u, v \in H$ and $x, y \in A$. 

\end{mydef}

The most well-known duality relation in the theory of quantum groups is between the Hopf algebras $U_q(\widehat{\mathfrak{sl}_2})$ and $\widehat{SL_q(2)}$. In this paper we will define a dual version of this duality relation.

 \subsection{$U_q(\widehat{\mathfrak{sl}_2})$}

In this section we will define and review some standard facts about $U_q(\widehat{\mathfrak{sl}_2})$. 

For $q$ a non-zero complex number and n a positive integer we define the quantum integers $[n]_q := \frac{q^n-q^{-n}}{q-q^{-1}}$. Let $[n]_q! := [1]_q [2]_q ... [n]_q$ and ${n \choose m}_q := \frac{[n]_q}{[m]_q [n-m]_q}$.

$U_q(\widehat{\mathfrak{sl}_2})$ is a quasitriangular Hopf algebra generated by the elements $K_i^{\pm}, e_i, f_i$, $i \in \{0,1 \}$ subject to the following relations:

\begin{itemize}
\item $K_i K_i^{-1} = 1=K_i^{-1} K_i$ \,\,\, $K_i K_j = K_j K_i$,
\item $K_i e_j K_i^{-1} = q^{a_{ij}} e_j$ \,\,\, $K_i f_j K_i^{-1} = q^{-a_{ij}} f_j$,
\item $e_i f_j - f_j e_i = \delta_{i,j} \frac{K_i -K_i^{-1}}{q-q^{-1}}$,
\item $\sum_{r=0}^{1-a_{ij}} {1-a_{ij} \choose r}_q e_i^{1-a_{ij}-r} e_j e_i^{r} = 0$  when $i \neq j$,
\item $\sum_{r=0}^{1-a_{ij}} {1-a_{ij} \choose r}_q f_i^{1-a_{ij}-r} f_j f_i^{r} = 0$  when $i \neq j$.
\end{itemize}
where the Cartan matrix of $\widehat{\mathfrak{sl}_2}$ is $A =  \left ( \begin{array}{cc}
a_{00} & a_{01} \\
a_{10} & a_{11} \\ \end{array} \right ) = \left ( \begin{array}{cc}
2 & -2 \\
-2 & 2 \\ \end{array} \right ) $.

The comultiplication, counit and antipode structure can be defined on the generators as follows:

\begin{itemize}
\item $\Delta (K_i) = K_i \otimes K_i$,
\item $\Delta (e_i) = e_i \otimes K_i + 1\otimes e_i$, $\Delta(f_i) = f_i \otimes 1 + K_i^{-1}\otimes f_i$,
\item $\epsilon(K_i) = 1$, $\epsilon(e_i) = 0 = \epsilon(f_i)$,
\item $S(K_i) = K_i^{-1}$, $S(e_i) = -e_i K_i^{-1}$, $S(f_i) = -K_i f_i$.
\end{itemize}

Finite dimensional modules of $U_q(\widehat{\mathfrak{sl}_2})$ on which $K_1$, $K_0$ act semisimply and the product $K_1 K_0$ acts as the identity are called type $\mathbf{1}$ modules. 

For every non-negative integer $r$ and complex number $a \in \mathbb{C}^*$, there is an $r+1$ dimensional irreducible $U_q(\widehat{\mathfrak{sl}_2})$ module $V_a(r)$ with basis $\{ v_0, ..., v_r \}$. The action of the generators is given below:

\begin{equation}\label{eq:module}
\begin{split}
K_1 v_j = q^{r-2j}  v_j, \,& \,\, K_0 v_j = q^{2j-r}  v_j  \\
e_1 v_j = [r-j+1]_q v_{j-1}, \,& \,\, f_1 v_j = [j+1]_q v_{j+1} \\
e_0 v_j = q^{-1} a [j+1]_q v_{j+1}, \,& \,\, f_0 v_j = q a^{-1} [r-j+1]_q v_{j-1}
\end{split}
\end{equation}

The module above is called an evaluation module. It  can be thought of as the pullback of the standard $n+1$ dimensional representation of $U_q(\mathfrak{gl}_2)$ by an evaluation morphim $ev: U_q(\widehat{\mathfrak{sl}_2}) \to U_q(\mathfrak{gl}_2)$ discovered by Jimbo, see $\cite{ChariPressley}$ Proposition 12.2.10.

Chari and Pressley $\cite{ChariPressley1}$ studied finite dimensional type $\mathbf{1}$ modules of $U_q(\widehat{\mathfrak{sl}_2})$ when $q$ is not a root of unity. They proved the following: 

\begin{mythm}\label{chari} (Chari and Pressley)
Every finite dimensional irreducible type $\mathbf{1}$ representation of $U_q(\widehat{\mathfrak{sl}_2})$ is isomorphic to a tensor product of evaluation representations:
\[
V_{a_1} (r_1) \otimes V_{a_2}(r_2) \otimes ... \otimes V_{a_n} (r_n)
\]
\end{mythm}

\begin{proof}
See Proposition 12.2.15 in $\cite{ChariPressley}$.
\end{proof}

In the same paper they also prove several other important facts about evaluation modules. For example they show that $V^*_{a}(n)$ is isomorphic as an $U_q(\widehat{\mathfrak{sl}_2})$ module to $V_{q^2a}(n)$ and they give conditions for when the tensor product above is $not$ irreducible.

\subsection{The FRT construction}

Given a vector space $V$, if $R$ is a solution of the Yang-Baxter equation (YBE), it was shown by Faddeev, Reshetkhin and Takhtajan in $\cite{Faddeev}$ that you can construct a coquasitriangular bialgebra that has $V$ as a comodule. The method is commonly referred to as the FRT construction.

\begin{mythm} (\textnormal{Faddeev, Reshetikhin and Takhtajan})
Let $R \in \textnormal{End}(V \otimes V)$ be a solution of the YBE. Then there exists a coquasitriangular bialgebra $A_R$ that has $V$ as a comodule. 
\end{mythm} 

Let $n$ be the dimension of $V$, $v_i$ a basis of $V$. $A_R$ is the unital algebra generated by elements $t_{ij}$ with $1 \leq i, j \leq n$ subject to the relations 

\begin{equation}\label{RTT}
R T_1 T_2 = T_2 T_1 R
\end{equation}
where $T$ is the $n$ by $n$ matrix with entries $t_{ij}$, $T_1 = T \otimes I$, $T_2 = I \otimes T$ and $I$ is the identity matrix.

The coalgebra structure is given by the following formulas:
\begin{equation*}
\begin{split}
\Delta (t_{ij}) = \sum_k & t_{ik} \otimes t_{kj} \\
\epsilon (t_{ij}) =& \delta_{ij}
\end{split}
\end{equation*}

$V$ becomes an $A_R$-comodule via the coaction:
\begin{equation*}
\Delta_V(v_i) = \sum_j t_{ij} \otimes v_j
\end{equation*}

One can then show the following fact:

\begin{mypro}
The condition for $\tau R \in \textnormal{End}(V \otimes V)$ to be an $A_R$-comodule homomorphism is $T_1 T_2 R = R T_2 T_1$.  
\end{mypro}

\begin{proof}
$\tau R$ is an $A_R$-comodule homomorphism if $\Delta_{V \otimes V}\circ (\tau R) = (1 \otimes (\tau R)) \circ \Delta_{V \otimes V}$. A short, but tedious computation shows that this is equivalent to $R T_1 T_2 = T_2 T_1 R$.
\end{proof}

$A_R$ has the structure of a coquasitriangular bialgebra. $\mathcal{R}: A_R \otimes A_R \to \mathbb{C}$ is given by the following formula on the generators of $A_R$:
\begin{equation*}
\mathcal{R}(t_{ij} \otimes t_{pq}) = R^{qj}_{pi}
\end{equation*}
where we use the following notational convention: $R (v_i \otimes v_j) = \sum_{k,l} R^{kl}_{ij} v_k \otimes v_l$.

We can then expand this formula to higher order terms of $A_R$ by using the second and third properties of the $R$ matrix in a coquasitriangular bialgebra.

\subsection{Corepresentations of $SL_q(n)$}

The solution to the YBE corresponding to the quantum group $U_q(\mathfrak{sl}_n)$ in the standard representation is $\cite{Jimbo1}$:
\begin{equation}
R = \sum_i q e_{ii} \otimes e_{ii} + \sum_{i \neq j} e_{ii} \otimes e_{jj} + (q - q^{-1}) \sum_{i>j} e_{ij} \otimes e_{ji}
\end{equation}

One can construct a quasitriangular bialgebra $A_R$ using this solution of the YBE. $A_R$ is a quantization of the ring of coordinate functions on $M_n(\mathbb{C})$. In $\cite{Faddeev}$, the quantum determinant is introduced. It has the following formula:
\begin{equation*} 
det_q = \sum_{\sigma \in S_n} (-q)^{l(\sigma)} t_{1 \sigma(1)} t_{2 \sigma(2)} ... t_{n \sigma(n)}
\end{equation*}
where $l(\sigma)$ is the length of the permutation $\sigma$.

$det_q$ is a central, group-like element in $A_R$. Define $SL_q(n)$ as the quotient algebra of $A_R$ mod the ideal generated by $det_q-1$.  $SL_q(n)$ is a coquasitriangular bialgebra as it is a quotient of $A_R$. It is also rigid; the antipode is given by the formula:
\begin{equation*}
S(t_{ij}) = (-q)^{i-j} \tilde t_{ji}
\end{equation*}
where $\tilde t_{ij} = \sum_{\sigma \in S_{n-1}} (-q)^{l(\sigma)}t_{1\sigma(1)}...t_{i-1 \sigma(i-1)} t_{i+1 \sigma(i+1)}... t_{n \sigma(n)} $. 

We can characterize the finite dimensional comodules of $SL_q(n)$ by defining a theory of highest weight comodules. This was done by Parshall and Wang in $\cite{ParshallWang}$. Each irreducible comodule $V$ is generated by a highest weight vector $v_+$. 

In the special case where $n=2$, $SL_q(2)$ will have one $n$ dimensional corepresentation up to isomorphism for each positive integer $n$ which we'll denote $V(n-1)$. $V(0)$ has basis $v$ and coaction $v \to 1 \otimes v$. $V(1) = V$. For $m \geq 2$, $V(m)$ will be a subcorepresentation of $V^{\otimes m}$. $V(m)^* \simeq V(m)$ for all non-negative $m$. This isomorphism can be deduced from the isomorphism in the case $m=1$ presented above. There is a duality relation between $U_q(\mathfrak{sl}_2)$ and $SL_q(2)$.

In this article we will construct $\widehat{SL_q(2)}$, the affine equivalent of $SL_q(2)$. We will that it has with respect to $U_q(\widehat{\mathfrak{sl}_2})$ many of the properties that $SL_q(2)$ has with respect to $U_q(\mathfrak{sl}_2)$. We will briefly talk about the general $n$ case as well.

\section{ $\widehat{SL_q(2)}$}

\subsection{The parametrized FRT construction}

The construction in this subsection is a parametrized version of the FRT construction $\cite{Faddeev}$ inspired by some results in $\cite{Cotta-Ramusino}$.

Let $R(x)$ be a solution of the parametrized YBE with group $\Gamma$ and vector space $W$ of dimension $n$ with basis $\{ w_i \}$. The entries of $R(x)$ will be denoted by $R^{kl}_{ij}$, they are given by the formula $R(x) w_i \otimes w_j = R_{ij}^{kl} w_k \otimes w_k$. 

We define $A_R(\Gamma)$ as the bialgebra generated by elements $\{1, t_{ij}(x) \}$ for $1 \leq i,j \leq n, \forall x \in \Gamma$ mod the ideal $\mathcal{I}_R$ generated by the elements:

\begin{equation}\label{eq:FRT}
\sum_{k,l} (R^{ab}_{kl}(y x^{-1}) t_{ik}(x) t_{jl}(y) - R^{lk}_{ij} (yx^{-1}) t_{kb}(y) t_{la}(x) )
\end{equation}
for all $i,j,a,b \in \{1,...,n \}$ and all $x, y \in \Gamma$. 

The counit and comultiplication are given by the formulas:
\begin{equation*}
\begin{split}
\Delta(t_{ij} (x)) = \sum_k & t_{ik}(x) \otimes t_{kj}(x) \\
 \epsilon(t_{ij} (x)) = &\delta_{ij}
 \end{split}
\end{equation*}

For any $x \in \Gamma$, let $W_x$ be a copy of the vector space $W$ with the same basis as above. We can endow $W_x$ with an $A_R(\Gamma)$ comodule structure as follows:
\begin{equation}\label{eq:standardcomodule}
\Delta_{W_x} (w_i) = t_{ij} (x) \otimes w_j
\end{equation}

\begin{mypro}
The map $\tau R(yx^{-1}) : W_x \otimes W_y \to W_y \otimes W_x$ is an $A_R(\Gamma)$-comodule homomorphism.  
\end{mypro}

\begin{proof}
Showing that $\tau R(yx^{-1})$ is a comodule homomorphism is equivalent to showing $(1 \otimes \tau R(yx^{-1})) \Delta_{W_x \otimes W_y} = \Delta_{W_y \otimes W_x} \tau R(y x^{-1})$. A short computation shows that this is equivalent to the element in $\mathcal{I}_R$ written in equation $\ref{eq:FRT}$ being $0$.

\end{proof}

At this point a remark is necessary. It is known that if $V_x, V_y$ are two dimensional evaluation modules for $U_q(\widehat{\mathfrak{sl}_2})$, then $\tau R(xy^{-1})$ is a comodule map between $V_x \otimes V_y$ and $V_y \otimes V_x$ and not $\tau R(yx^{-1})$. In our case we use $\tau R(yx^{-1})$ because we work in the dual setting; we will see that duals of comodules of the object we build will be modules of $U_q(\widehat{\mathfrak{sl}_2})$ and the functor taking one to the other is contravariant. Therefore the comodule map $\tau R(yx^{-1}): W_x \otimes W_y \to W_y \otimes W_x$ will correspond to the module map $\tau R(yx^{-1}): V_y \otimes V_x \to V_x \otimes V_y$ which is an $U_q(\widehat{\mathfrak{sl}_2})$ homomorphism.

From now on we will denote $\Delta_{W_x}$ by $\Delta$ similar to the comultiplication on $A_R(\Gamma)$. One should be able to differentiate the two from context. 

\subsection{A new Hopf algebra: $\widehat{SL_q(2)}$}

Take $R_q(x)$ to be

\begin{equation}
R_q(x)=\left ( \begin{array}{cccc}
q-xq^{-1} & 0 & 0 & 0 \\
0 & 1-x & x(q-q^{-1}) & 0 \\
0 & q-q^{-1} & 1-x & 0 \\
0 & 0 & 0 & q-xq^{-1} \end{array} \right )
\end{equation}
in the basis of $\{w_1 \otimes w_1, w_2 \otimes w_1, w_1 \otimes w_2, w_2 \otimes w_2 \}$. Note that this is (up to a factor) the action of the universal $\mathcal{R}$-matrix of $U_q(\widehat{\mathfrak{sl}_2})$ on tensor products of two dimensional evaluation modules $V_a \otimes V_b$, where $x = \frac{a}{b}$.

Let 
\[T(x) = \left ( \begin{array}{cc}
t_{11}(x)& t_{12}(x) \\
t_{21}(x) & t_{22}(x) \\ \end{array} \right ) \textnormal{for } x \in \Gamma.
\] 

By the method described above we obtain a bialgebra $A_{R_q}(\Gamma)$ generated by the elements $1, t_{11}(x), t_{12}(x), t_{21}(x), t_{22}(x)$ modulo the ideal generated by elements in equation $\ref{eq:FRT}$ for all $x, y \in \Gamma$. We write equation $\ref{eq:FRT}$ in matrix form:

\footnotesize

\begin{equation*}
\begin{split}
 \kern -10pt \left ( \begin{array}{cccc}
q-\frac{y}{x} q^{-1} &  &  &  \\
 & \kern -30pt 1- \frac{y}{x} & \frac{y}{x}(q-q^{-1}) &  \\
 & \kern -30pt q-q^{-1} & 1-\frac{y}{x} &  \\
 &  &  &  \kern -20pt q-\frac{y}{x}q^{-1} \end{array} \right )   
\left ( \begin{array}{cccc}
t_{11}(x)t_{11}(y) & t_{21}(x)t_{11}(y) & t_{11}(x)t_{21}(y) & t_{21}(x)t_{21}(y) \\
t_{12}(x)t_{11}(y) & t_{22}(x)t_{11}(y) & t_{12}(x)t_{21}(y) & t_{22}(x)t_{21}(y) \\
t_{11}(x)t_{12}(y) & t_{21}(x)t_{12}(y) & t_{11}(x)t_{22}(y) & t_{21}(x)t_{22}(y) \\
t_{12}(x)t_{12}(y) & t_{22}(x)t_{12}(y) & t_{12}(x)t_{22}(y) &  t_{22}(x)t_{22}(y) \end{array} \right ) =  \\
\left ( \begin{array}{cccc}
t_{11}(y) t_{11}(x) & t_{11}(y) t_{21}(x) & t_{21}(y) t_{11}(x) & t_{21}(y) t_{21}(x) \\
t_{11}(y) t_{12}(x) & t_{11}(y) t_{22}(x) & t_{21}(y) t_{12}(x) & t_{21}(y) t_{22}(x) \\
t_{12}(y) t_{11}(x) & t_{12}(y) t_{21}(x) & t_{22}(y) t_{11}(x) & t_{22}(y) t_{21}(x) \\
t_{12}(y) t_{12}(x) & t_{12}(y) t_{22}(x) & t_{22}(y) t_{12}(x) & t_{22}(y) t_{22}(x) \end{array} \right )   
\left ( \begin{array}{cccc}
q-\frac{y}{x} q^{-1} &  &  &  \\
 & \kern -30pt 1- \frac{y}{x} & \frac{y}{x}(q-q^{-1}) &  \\
 & \kern -30pt q-q^{-1} & 1-\frac{y}{x} &  \\
 &  &  &  \kern -20pt q-\frac{y}{x}q^{-1} \end{array} \right )   
\end{split}
\end{equation*}

\large

Notice that $R_q(q^2)$ has rank $1$ and $R_{q} (q^{-2})$ has rank $3$. Otherwise the matrix $R_{q}(x)$ is invertible. By plugging in $y=q^2x$ in the equation above and expanding, one gets the following ``commutation relations'': 

\begin{equation}\label{eq:commutationformq2hat}
\begin{split}
t_{12}(x) t_{11}(q^2x)=& q t_{11}(x)t_{12}(q^2x) \\
t_{21}(q^2x)t_{11}(x) =& q t_{11}(q^2x)t_{21}(x) \\
t_{22}(q^2 x) t_{12}(x)=&  q t_{12}(q^2x) t_{22}(x)\\
t_{22}(x) t_{21}(q^2x) =& q t_{22}(x) t_{22}(q^2x) \\
t_{22}(x)t_{11}(q^2x) -  q t_{21}(x)t_{12}(q^2x) =& t_{11}(q^2x)t_{22}(x) - q^{-1} t_{21}(q^2x)t_{12}(x)   = \\
= t_{22}(q^2x) t_{11}(x) - q  t_{12}(q^2x) t_{21}(x)  =& t_{11}(x) t_{22}(q^2x) - q^{-1} t_{12}(x) t_{21}(q^2x) := \textnormal{det}_q(qx)
\end{split}
\end{equation}

The last set of four equalities are used to define the affine version of the quantum determinant.

A short computation shows that the quantum determinant is group-like:
\begin{gather*}
\Delta (\textnormal{det}_q(qx)) = \Delta (t_{11}(x) t_{22}(q^2x) - q^{-1} t_{12}(x) t_{21}(q^2x)) = \Delta(t_{11}(x)) \Delta(t_{22}(q^2x)) -  \\
q^{-1} \Delta (t_{12}(x)) \Delta (t_{21}(q^2x))
=(t_{11}(x) \otimes t_{11}(x) + t_{12}(x) \otimes t_{21}(x))(t_{21}(q^2x) \otimes t_{12}(q^2x) + \\ t_{22}(q^2x) \otimes t_{22}(q^2x)) - q^{-1} (t_{11}(x) \otimes t_{12}(x) + t_{12}(x) \otimes t_{22}(x))(t_{21}(q^2x) \otimes t_{11}(q^2x) + \\ t_{22}(q^2x) \otimes t_{21}(q^2x)) = 
t_{11}(x)t_{21}(q^2x) \otimes t_{11}(x)t_{12}(q^2x) +  t_{11}(x)t_{22}(q^2x) \otimes t_{11}(x)t_{22}(q^2x) + \\ t_{12}(x) t_{21}(q^2x) \otimes t_{21}(x) t_{12}(q^2x) + t_{12}(x)t_{22}(q^2x) \otimes t_{21}(x) t_{22}(q^2x)  - \\ q^{-1} t_{11}(x)t_{21}(q^2x) \otimes t_{12}(x)t_{11}(q^2x) - q^{-1} t_{11}(x)t_{22}(q^2x) \otimes t_{12}(x)t_{21}(q^2x) - \\ q^{-1} t_{12}(x)t_{21}(q^2x) \otimes t_{21}(x)t_{12}(q^2x) - q^{-1} t_{12}(x)t_{22}(q^2x) \otimes t_{22}(x)t_{21}(q^2x) = \textnormal{det}_q(qx) \otimes \textnormal{det}_q(qx)
\end{gather*}

If we now quotient $A_R(\Gamma)$ by the ideal generated by the elements $\textnormal{det}_q(x)-1$ for all $x \in \Gamma$ and call it $\widehat{SL_q(2)}$, we can endow this bialgebra with an antipode:
 
\begin{equation}\label{eq:hatantipode}
\begin{split}
S(t_{11}(x)) = & t_{22}(q^{2}x)\\
S(t_{12}(x)) =& -qt_{12}(q^{2}x) \\
S(t_{21}(x)) =& -q^{-1} t_{21}(q^{2}x)\\
S(t_{22}(x)) =& t_{11}(q^{2}x)
\end{split}
\end{equation}

\begin{mythm}
$\widehat{SL_q(2)}$ is a Hopf algebra with the antipode defined above.
\end{mythm}

\begin{proof}
In order to make sure that the formula for the antipode is correct, we just need to check the following relations:

\begin{gather*}
\left ( \begin{array}{cccc}
t_{11}(x) & t_{12}(x)  \\
t_{21}(x) & t_{22}(x)   \end{array} \right ) 
\left ( \begin{array}{cccc}
S(t_{11}(x)) & S(t_{12}(x))  \\
S(t_{21}(x)) & S(t_{22}(x))   \end{array} \right ) = I \\
\left ( \begin{array}{cccc}
S(t_{11}(x)) & S(t_{12}(x))  \\
S(t_{21}(x)) & S(t_{22}(x))   \end{array} \right ) \left ( \begin{array}{cccc}
t_{11}(x) & t_{12}(x)  \\
t_{21}(x) & t_{22}(x)   \end{array} \right ) = I
\end{gather*}

By writing down the values of the antipode according to formula  $\ref{eq:hatantipode}$ and using the fact that $\det_q(qx)=1$ we can show that $S$ is indeed the antipode for $\widehat{SL_q(2)}$. 

\end{proof}

\subsection{Duality between $U_q(\widehat{\mathfrak{sl}_2})$ and $\widehat{SL_q(2)}$ }

Let $T(x) = \left ( \begin{array}{cc}
t_{11}(x)& t_{12}(x) \\
t_{21}(x) & t_{22}(x) \\ \end{array} \right )$ be defined as above. By $\braket{x, T(x)} $ we mean $ \left ( \begin{array}{cc}
\braket{x,t_{11}(x)}& \braket{x,t_{12}(x)} \\
\braket{x,t_{21}(x)} & \braket{x,t_{22}(x)} \\ \end{array} \right )$. The following theorem relates $U_q(\widehat{\mathfrak{sl}_2})$ and $\widehat{SL_q(2)}$.

\begin{mythm} 
There is a duality relation $\braket{ , }$ between $U_q(\widehat{\mathfrak{sl}_2})$ and $\widehat{SL_q(2)}$ that is given on generators by the following formulas:
\begin{equation}\label{eq:duality}
\begin{split}
\braket{K_1,T(x)} = \left ( \begin{array}{cc}
q & 0 \\
0 & q^{-1} \\ \end{array} \right ) \,\,\,\,\,\, \braket{K_0,T(x)} = \left ( \begin{array}{cc}
q^{-1} & 0 \\
0 & q \\ \end{array} \right )\\
\braket{e_1,T(x)} = \left ( \begin{array}{cc}
0 & 1 \\
0 & 0 \\ \end{array} \right ) \,\,\,\,\,\, \braket{e_0,T(x)} = \left ( \begin{array}{cc}
0 & 0 \\
q^{-1}x & 0 \\ \end{array} \right ) \\
\braket{f_1,T(x)} = \left ( \begin{array}{cc}
0 & 0 \\
1 & 0 \\ \end{array} \right ) \,\,\,\,\,\, \braket{f_0,T(x)} = \left ( \begin{array}{cc}
0 & qx^{-1} \\
0 & 0 \\ \end{array} \right ) \\
\braket{1,T(x)} = \left ( \begin{array}{cc}
1 & 0 \\
0 & 1 \\ \end{array} \right ) \,\,\,\,\,\, \braket{a,1} = \epsilon(a), \forall a \in U_q(\widehat{\mathfrak{sl}_2})
\end{split}
\end{equation}

\end{mythm}

\begin{proof} 
Since we defined the duality on generators, the relations in definition $\ref{def:duality}$ will hold. One thing that needs checking is the fact that the duality relation is well-defined, namely the fact the $\braket{a,t} = 0$  for every element $t$ of the form in equation $\ref{eq:FRT}$ and $\braket{a,\det_q (x)} = \epsilon(a)$. 

The second equality is easier. We have to prove it for $a$ a generator of $U_q(\widehat{\mathfrak{sl}_2})$ (i.e. $K_i$, $e_i$ and $f_i$) and then for products of such generators we can use the fact that $\det_q(x)$ is group-like and therefore $\braket{a_1 a_2, det_q(x) } = \braket{a_1,\det_q(x)} \braket{a_2, \det_q(x)}$. 

The first equality is significantly harder from a computational point of view. The idea is to check that $\braket{a,t}=0$ for every element $t$ described above and $a = e_0^{i_1}e_1^{i_2}f_0^{i_3}f_1^{i_4}K_0^{i_5}K_1^{i_6}$. $t$ is a product of degree two of generators of $\widehat{SL_q(2)}$, so we first find $\Delta (a)$. Several cases need to be worked out independently (for example if any of the $i_1$, $i_2$, $i_3$ and $i_4$ are greater than $2$, it follows that the bracket is $0$ due to the fact that their squares act as $0$ on the two dimensional evaluation module). 

We will compute a very simple case to try to convince the reader that this relation holds. We will show that $\braket{a,t} = 0$ for
\[
t = (q-\frac{y}{x}q^{-1})t_{21}(x) t_{11}(y) - (1-\frac{y}{x})t_{11}(y) t_{21}(x) - (q-q^{-1}) t_{21}(y) t_{11}(x).
\]
 and $a \in \{e_i, f_i, K_i \}$.
 
Because $\Delta(K_i) = K_i \otimes K_i$ and $\braket{K_i,t_{21}(x)} = 0$ it follows that  $\braket{K_i, t}=0$. $e_1$'s and $f_0$'s bracket with $t_{21}(x)$ and $t_{11}(x)$ are $0$, and since $\Delta (e_1) = e_1 \otimes K_1 + 1\otimes e_1$ and $\Delta(f_0) = f_0 \otimes 1 + K_0^{-1}\otimes f_0$ we get $0$ again. 

In the remaining cases we write the comultiplication and compute the bracket: $\Delta (e_0) = e_0 \otimes K_0 + 1\otimes e_0, \Delta(f_1) = f_1 \otimes 1 + K_1^{-1}\otimes f_1$.

\[
\braket{e_0, t} = q^{-1} (  (q-\frac{y}{x}q^{-1})(x q^{-1}) - (1-\frac{y}{x})y - (q-q^{-1})(y q^{-1}) )= 0
\]
 \[
\braket{f_1, t} = (q-\frac{y}{x}q^{-1}) - (1-\frac{y}{x})q^{-1} - (q-q^{-1}) = 0 
 \]

\end{proof}

 \subsection{Evaluation comodules of $\widehat{SL_q(2)}$}

Let $W$ be a finite dimensional comodule of $\widehat{SL_q(2)}$, $w \in W$. We denote the coaction by $w \to w_{(0)} \otimes w_{(1)} $. One can show that the dual of $W$, which we'll denote $\bar W = V$ is now a module of $U_q(\widehat{\mathfrak{sl}_2})$. Let $x \in U_q(\widehat{\mathfrak{sl}_2}), \bar w \in \bar W$. The action will be given by 

\begin{equation}\label{eq:action}
x \cdot \bar w (w) = \braket{x, w_{(0)}} \bar w (w_{(1)})
\end{equation}

The coevaluation $\Delta$ is a map from $W$ to $\widehat{SL_q(2)} \otimes W$. Given a basis $\{w_i \}$ of $W$, define $\alpha_{jl} \in \widehat{SL_q(2)}$ such that $\Delta (w_j) = \alpha_{jl} \otimes w_l$. Using equation $\ref{eq:action}$ we can now write the action of $x$ on $V$ as follows:
\[
x \cdot \bar w_j = \braket{x, \alpha_{lj}} \bar w_l
\]

We know that $U_q(\widehat{\mathfrak{sl}_2})$ has an evaluation module $V_a(r)$ of dimension $r+1$ for every $a \in \mathbb{C}^*$ and $r$ non-negative integer (note that all one dimensional modules are in fact the same regardless of what $a$ is). See equation $\ref{eq:module}$ for the action of the generators of $U_q(\widehat{\mathfrak{sl}_2})$ on $V_a(r)$. We will now build evaluation comodules $W_a(r)$ for $\widehat{SL_q(2)}$.

If $r=0$, then $W(0)$ is the one dimensional comodule with the coaction $v \to 1 \otimes v$. If $r=1$, $W_{a}(1)$ is the comodule $W_a$ defined in equation $\ref{eq:standardcomodule}$.

A basis of $W_{a_1} \otimes W_{a_2} \otimes ... \otimes W_{a_n}$ is given by $w_{i_1} \otimes w_{i_2} \otimes ... \otimes w_{i_n} := w_{i_1,i_2,...,i_n}$, where $i_j \in \{1,2\}$.

$W_{a}(r)$ will be the subcomodule of $W_{q^{-r+1}a} \otimes W_{q^{-r+3}a} \otimes ... \otimes W_{q^{r-1}a}$ generated by the ``highest weight vector'' $w_{1,1,..,1}$. It will have basis $\{u_{j} \}, 0 \leq j \leq r$ given by the following formula:
\[
u_{j} = \sum_{i_k \in \{1,2 \}, \sum_{k} i_k = r+j} g(i_1,i_2,...,i_r) w_{i_1,i_2,...,i_r}
\]
where $g(i_1,i_2,...,i_r)$ is $q^p$ with $p$ being the sum over all $i_{m}=2$ of the number of $i_k =1$ that are to the right of that $i_{m}=2$ in the sequence $\{ i_1, i_2,  ... , i_n \}$. 

For example $u_0 = w_{1,1,...,1}$ and $u_{1} = w_{1,1,..,2}+qw_{1,...,2,1} +...+ q^{r-1} w_{2,1,...,1}$. 

The comodule structure on $W_{a}$ is given by $ \Delta (u_i) = \sum_j \alpha_{ij} \otimes u_j$, where 

\begin{equation}\label{eq:evaluationcomodule}
\alpha_{ij} = \sum_{i_k \in \{1,2 \}, \sum_{k} i_k = r+i} g(i_1,i_2,...,i_r) t_{i_1 j_1} (q^{-r+1}a) ... t_{i_r j_r} (q^{r-1}a)
\end{equation}
where $j_k=1$ for $k \leq r-j$ and $j_k =2$ otherwise. We skip the proof of the fact that this is indeed a module, and that it's irreducible, but note that it involves repeated use of the ``commutation relations'' in equation $\ref{eq:commutationformq2hat}$.

\begin{mythm}
$\widehat{SL_q(2)}$ has an irreducible comodule $ W_a(r)$ such that $\bar W_a(r)$ is isomorphic to the $U_q(\widehat{\mathfrak{sl}_2})$ module $V_a(r)$.
\end{mythm}

\begin{proof}

 A simple computation using equation $\ref{eq:action}$ and the duality relation $\ref{eq:duality}$ shows that $\bar W_{a}(1)$ is isomorphic to $V_{a}(1)$. 

For $r \geq 2$ let $\Delta^{r-1}: U_q(\widehat{\mathfrak{sl}_2}) \to U_q(\widehat{\mathfrak{sl}_2})^{\otimes r} $ be defined as the composition $(\Delta \otimes I \otimes .. \otimes I)..(\Delta \otimes I) \Delta$ where we have $r-1$ terms in the composition. This is an asymmetry in our definition because $\Delta$ act on the left side; it is taken care of by coassociativity. The following formulas are well-known:

\[
\Delta^{r-1}(K_i) = K_i \otimes K_i \otimes ... \otimes K_i
\]
\[
\Delta^{r-1}(e_i) = 1 \otimes ... \otimes 1 \otimes e_i + 1 \otimes .. \otimes e_i \otimes K_i + .. + e_i \otimes K_i \otimes ... \otimes K_i
\]
\[
\Delta^{r-1} (f_i) = f_i \otimes 1 \otimes ... \otimes 1 + K_i^{-1} \otimes f_i \otimes ... \otimes 1 + .. + K_i^{-1} \otimes ... \otimes K_i^{-1}  \otimes f_i
\]

We are now ready to prove the following theorem relating comodules of $\widehat{SL_q(2)}$ and modules of $U_q(\widehat{\mathfrak{sl}_2})$. 

The generators of $U_q(\widehat{\mathfrak{sl}_2})$ will act on $\bar W_a(r)$ via the formula mentioned at the beginning of the subsection: $x \cdot \bar u_i = \braket{x, \alpha_{ji}} \bar u_j$. So we are only interested in the coefficients $\braket{x, \alpha_{ji}} $ for $\alpha_{ji}$ defined in equation $\ref{eq:evaluationcomodule}$. For $K_i$ the coefficients $\braket{K_i, t_{kl}(x)}$ are non-zero only when $\delta_{kl}=1$. It is not too hard to see that:
\[
\braket{K_1, \alpha_{ii}} = q^{r-2i}, \braket{K_0,\alpha_{ii}} = q^{2i-r}
\]

For $e_1$, note that $\braket{e_1, t_{kl}(x)}$ is non-zero only when $k=1, l=2$. Looking at the formula for $\Delta^{r-1} (e_1)$ we conclude that the only non-zero coefficients will be $\braket{e_1,\alpha_{j-1, j} }$. Only $j$ terms in the expression of  $\alpha_{j-1, j}$ will be non-zero under the bracket with $e_1$, namely 
\[
\braket{e_1, g(1,1,...,i_r) t_{1 1} (q^{-r+1}a) ...t_{1,1}(q^{r-2j-1}a)  t_{i_{r-j+1} 2}(q^{r-2j+1}a) ... t_{i_r 2} (q^{r-1}a)}
\]
where only one of the $i_k, k \in [ r-j+1, r]$ is $1$ and the rest are $2$. The value of the term above will be be $q^{1}$. Summing over all possible terms we get $\braket{e_1,\alpha_{j-1, j} } = q^{-j+1}+q^{-j+3}+..+q^{j-1} = [j]_q$. This means that 
\[
e_1 \bar u_j = [j]_q \bar u_{j-1}
\]

In a similar fashion we obtain

\[
e_0 \bar u_j = q^{-1}a [r-j]_q \bar u_{j+1}
\]
\[
f_1 \bar u_j =  [r-j]_q \bar u_{j+1}
\]
\[
f_0 \bar u_j =   qa^{-1} [j]_q \bar u_{j-1}
\]

By making a change of basis in $\bar W_a(r)$ that takes $\bar u_j \to {r \choose j}_q \bar u_j$ we get the exact same action of the generators on $\bar W_a(r)$ as on $V_a(r)$, see equation $\ref{eq:module}$.

\end{proof}

\subsection{Dual of an evaluation comodule}

It is well know that given a Hopf algebra $H$ and a module $V$, then $V^*$ will also be a module via the action $x \cdot v^*(v) = v^*(S(x)v)$. One can write this action diagrammatically and ``reverse all arrows'' in order to come up with a similar formula for the comodules of a Hopf algebra. Here we skip the details and write down the formula directly. If $W$ is a comodule of $H$ such that the coaction takes $w_i \to \alpha_{ij} \otimes w_j$ with $\alpha_{ij} \in H$, then its dual $W^*$ is a comodule of $H$ via the coaction $w^*_i \to S^{-1}(\alpha_{ji}) \otimes w^{*}_{j}$, where $w^*_i(w_j) = \delta_{ij}$.

\begin{mypro}
The dual of $W_{a}(n)$ (as an $\widehat{SL_q(2)}$ comodule) is isomorphic to $W_{q^{-2}a}(n)$. 
\end{mypro}

\begin{proof}
When $n=1$, one can prove this by writing down the formula above and coming up with an explicit isomorphism. An interesting fact is that one can also look at the homomorphism $\tau R(q^2): W_{q^{-2}a} \otimes W_{a} \to W_{a} \otimes W_{q^{-2}a}$ which has rank $1$ and notice that it can be interpreted as an evaluation map onto its image. $W_{q^{-2}a} \otimes W_{a}$ has a three dimensional subcomodule (the image of $\tau R(q^{-2})$), we can quotient by that subcomodule and treat the map $\tau R(q^2): W_{q^{-2}a} \otimes W_{a} \to W_{a} \otimes W_{q^{-2}a}$ as an coevaluation map. One can then show that these maps satisfy the necessary axioms for evaluation and coevaluation maps (for example $ (I \otimes ev)(coev \otimes I) = I$). This will then produce an isomorphism between $W_{q^{-2}a}(n)$ and $W_{a}(n)^*$.

For general $n$, one can define the following maps: the evaluation map $ev: W_{q^{-2}a}(n) \otimes W_{a}(n) \to k$ given by 

\[
ev(w_i \otimes w_j) =  \sum_{i,j} (-1)^{j} \delta_{i, r-j}  q^{r} q^{r-2}...q^{r-2(j-1)}   {r \choose j}^{-1}_q 
\]

and the coevaluation $coev : k \to W_{a}(n) \otimes W_{q^{-2}a}(n)$ given by 
\[
coev(1) = \sum_{i,j} \delta_{r-j,i} (-1)^{j}  q^{-r} q^{2-r}...q^{2(j-1)-r} {r \choose j}_q  w_{j} \otimes w_{i}
\]

One needs to show that these two maps satisfy the necessary axioms, namely 
\[
(I_{W_a(n)} \otimes ev)(coev \otimes I_{W_a(n)})=I_{W_a(n)}
\]
\[
(ev \otimes I_{W_{q^{-2}a}(n)})(I_{W_{q^{-2}a}(n)} \otimes coev) = I_{W_{q^{-2}a}(n)}
\]
where $I$ is the identity map. This is just an easy computation. Second thing that needs to be done is to show that these maps are $\widehat{SL_q(2)}$-comodule homomorphisms. We skip the details of this rather long calculation. 

\end{proof}

\subsection{A tensor product decomposition}

In $\cite{ChariPressley1}$ Chari and Pressley prove that $V_x(m) \otimes V_y(n)$ as a module of $U_q(\widehat{\mathfrak{sl}_2})$ is irreducible if and only if $\frac{x}{y} \neq q^{\pm (m+n-2p+2)}$ for any $0<p \leq \text{min} \{ m,n\}$. We prove a similar proposition for $\widehat{SL_q(2)}$:

\begin{mypro}\label{prop:irreducibility}
$W_{x}(m) \otimes W_{y}(n)$ is irreducible if and only if $\frac{x}{y} \neq q^{\pm (m+n-2p+2)}$ for any $0<p \leq \text{min} \{ m,n\}$. 
\end{mypro}

\begin{proof}
Let $U$ be a comodule of $\widehat{SL_q(2)}$ such that $u_i \to \alpha_{ij} \otimes u_j$, $\alpha_{ij} \in \widehat{SL_q(2)}$. Then $U$ will be a comodule of $SL_q(2)$ with coaction $u_i \to \bar \alpha_{ij} \otimes u_j$, where $\bar \alpha_{ij}$ is obtained from $\alpha_{ij}$ by replacing all $t_{ij}(x)$ with $t_ij \in SL_q(2)$. This makes sense only if replacing $t_{ij}(x)$ with $t_{ij}$ in the defining relations of $\widehat{SL_q(2)}$ would not create any inconsistencies. 

The defining relations of $\widehat{SL_q(2)}$ are equation $\ref{eq:FRT}$ and setting $\det_q(x) =1$. Doing the replacement in $\det_q(x)$ gives us $\det_q = 1 \in SL_q(2)$. Equation $\ref{eq:FRT}$ is equivalent to $\tau R_q(y^{-1}x)$ is comodule homomorphims. But $\tau R_q(y^{-1}x) = \tau R_q - y^{-1}x (R^{-1}_q) \tau$, where $R_q$ is the $R$-matrix corresponding to $U_q(\mathfrak{sl}_2)$. Because of that, $\tau R_q$ is a $SL_q(2)$-comodule homomorphism $: V \otimes V  \to V \otimes V$, and so is $(R^{-1}_q) \tau$ (basically the inverse). Since $SL_q(2)$ is defined in such a way that $\tau R_q$ is a homomorphism, there are no inconsistencies.

If $W_{x}(m) \otimes W_{y}(n)$ has a subcomodule $U$, then $U$ will also be a subcomodule of $W(m) \otimes W(n)$, where $W(r)$ is the $r+1$ dimensional comodule of $SL_q(2)$. But $W(m) \otimes W(n)$ splits just like it does for $U_q(\mathfrak{sl}_2)$, namely 

\begin{equation}\label{eq:CG}
W(m) \otimes W(n) \simeq W(m+n) \oplus ... \oplus W(|m-n|). 
\end{equation}

This will mean $U$ will be a direct sum of some of the summands in $\ref{eq:CG}$. One can then pick up the highest weight vector $\Omega_p$ in $W(m+n-2p)$ and can show by the way of computation that $\Omega_p$ will be part of a subcomodule of $W(m) \otimes W(n)$ not containing $\Omega_0$ if and only if $\frac{b}{a} = q^{-(m+n-2p+2)}$ for $0 < p \leq m,n$. 

One can then show that $W(m) \otimes W(n)$ has a subcomodule containing the highest weight vector in tensor product if and only if $\frac{b}{a} = q^{(m+n-2p+2)}$ for $0 < p \leq m,n$.

\end{proof}

Note that the argument we used in the proof above is basically the same argument as in Proposition 4.8 of $\cite{ChariPressley1}$.

\subsection{The duality relation revised}

In this section we assume the duality relation defined in $\ref{eq:duality}$ is non-degenerate. This is a nontrivial result as far as we can tell. We will prove a theorem based on this assumption that is meant to be taken as a conjecture. At the end of the subsection we discuss the implications of these results. 

\begin{mypro}
Let $W$ be an irreducible finite dimensional comodule of $\widehat{SL_q(2)}$. Then $\bar W$ is an irreducible module of $U_q(\widehat{\mathfrak{sl}_2})$. 
\end{mypro}

\begin{proof}
Assume $\bar W$ has a submodule $U$. Pick a basis $\bar w_1, .., \bar w_k$ of $U$ and extend it to a basis $\bar w_1, ..., \bar w_{k}, \bar w_{k+1}, ..., \bar w_n$ of $\bar W$. Let $w_i$ be the dual basis of $W$, so that we have $\bar w_j(w_i) = \delta_{ji}$.

Define $\alpha_{jl} \in \widehat{SL_q(2)}$ such that $\Delta (w_j) = \alpha_{jl} \otimes w_l$. As discussed above (see equation $\ref{eq:action}$), we can now write the action of $x$ on $V$ as follows:
\[
x \cdot \bar w_j = \braket{x, \alpha_{lj}} \bar w_l
\]

If $U$ is a submodule of $\bar W$ then this means that $x \cdot \bar w_j \in W$ for all $j \in \{ 1, ..., k\}$ which implies that $\braket{x, \alpha_{lj}}=0$ for all $l \in \{ k+1, ..., n \}, j \in \{ 1, ..., k \}$ and for all $x$. 

It then must follow that  $ \alpha_{lj}=0$ for all $l \in \{ k+1, ..., n \}, j \in \{ 1, ..., k \}$ because of the non-degeneracy of the duality form. 

Because of this, the span of all the $w_l, l \in \{ k+1, ..., n \}$ will form a subcomodule of $\widehat{SL_q(2)}$. We obtained a contradiction, therefore we are done. 
\end{proof} 

\begin{lemma}\label{lemma:uniqueness}
Let $W_1$ and $W_2$ be irreducible finite dimensional comodules of $\widehat{SL_q(2)}$ such that $\bar W_1$ and $\bar W_2$ are isomorphic as $U_q(\widehat{\mathfrak{sl}_2})$ modules. Then $W_1$ and $W_2$ are isomorphic as $\widehat{SL_q(2)}$ comodules.
\end{lemma}

\begin{proof}
Let $w^{(1)}_i$ be a basis of $W_1$ with $\Delta(w^{(1)}_i) = \sum_j \alpha_{ij} \otimes w^{(1)}_j$ and $w^{(2)}_i$ a basis of $W_2$ with $\Delta(w^{(2)}_i) = \sum_j \beta_{ij} \otimes w^{(2)}_j$, for $\alpha_{ij}, \beta_{ij} \in \widehat{SL_q(2)}$ such that the isomorphism $f$ between $\bar W_1$ and $\bar W_2$ takes $\bar w^{(1)}_i \to \bar w^{(2)}_i$. It follows that

\begin{gather*}
x \cdot w^{(1)}_i (w^{(1)}_k) =  \sum_j \braket{x, \alpha_{kj}} w^{(1)}_i (w^{(1)}_j) = \braket{x, \alpha_{ki}} \\
x \cdot w^{(2)}_i (w^{(2)}_k) =  \sum_j \braket{x, \beta_{kj}} w^{(2)}_i (w^{(2)}_j) = \braket{x, \beta_{ki}}
\end{gather*}

The fact that $f$ is a module homomorphism implies that if $x \cdot \bar w^{(1)}_i = \gamma_{ij} \bar w^{(1)}_j$, then $x \cdot \bar w^{(2)}_i = \gamma_{ij} \bar w^{(2)}_j$ for any $x \in U_q(\widehat{\mathfrak{sl}_2})$.  

We  know that
\begin{gather*}
x \cdot \bar  w^{(1)}_i (w^{(1)}_k) = \sum_j \gamma_{ij} \bar  w^{(1)}_j (w^{(1)}_k) = \gamma_{ik}
x \cdot \bar  w^{(2)}_i (w^{(2)}_k) = \sum_j \gamma_{ij} \bar  w^{(2)}_j (w^{(2)}_k) = \gamma_{ik} \\
\end{gather*}

It follows that $\braket{x, \alpha_{ki}} = \gamma_{ik}  = \braket{x, \beta_{ki}}$ for all $x$, which implies that $\alpha_{ki} = \beta_{ki}$ by the non-degeneracy of $\braket{,}$. We conclude that $f$ is a comodule isomorphism between $V$ and $W$.

\end{proof}

Because of the way the duality is defined, we can show that the $K_1 K_0 $ must act as the identity on any $U_q(\widehat{\mathfrak{sl}_2})$ module $\bar W$ obtained from a comodule $W$ of $\widehat{SL_q(2)}$.

\begin{lemma}
Let $W$ be a comodule of $\widehat{SL_q(2)}$, and $\bar W$ the associated module of $U_q(\widehat{\mathfrak{sl}_2})$. Then $K_1 K_0 $ acts as the identity on $\bar W$. 
\end{lemma}
\begin{proof}
This is due to the fact that $\braket{K_1 K_0  , t} = \epsilon(t) $ for all $t \in \widehat{SL_q(2)}$.
\end{proof}

Let $\mathcal{T}$ be the quotient of $\widehat{SL_q(2)}$ by setting $t_{12}(x)$ and $t_{21}(x)$ equal to $0$ for all $x \in \mathbb{C}^*$. Given a comodule $W$ of $\widehat{SL_q(2)}$ one can build a $\mathcal{T}$-comodule by the usual method. 

We say a comodule $W$ of $\widehat{SL_q(2)}$ is of type $\mathbf{1}$ if the coaction on the corresponding $\mathcal{T}$-comodule acts semisimply; namely, if $W$ has a basis $w_i$ such that the coaction acts as $w_i \to t_i \otimes w_i$ (note that we do not sum over $i$) for $t_i \in \mathcal{T}$. 

\begin{lemma}
Let $W$ be a type $\bf{1}$ comodule of $\widehat{SL_q(2)}$, and $\bar W$ the associated module of $U_q(\widehat{\mathfrak{sl}_2})$. Then $K_i $ acts semisimply on $\bar W$. 
\end{lemma}
\begin{proof}
Since $W$ is a type $\bf{1}$ comodule, there is a basis $w_i$ of $W$ such that the coaction on the $\mathcal{T}$-comodule $W$ is $w_j \to t_j \otimes w_j$ for some $t_i \in\mathcal{T}$. It follows that $K_i$ will act semisimply on $\bar W$, it will take $\bar w_j \to \braket{K_i,t_j} \bar w_j$. 
\end{proof}

\begin{mycon}
Every finite dimensional irreducible type $\bf{1}$ comodule of $\widehat{SL_q(2)}$ will be isomorphic to a tensor product of the form:
\[
W_{a_1} (r_1) \otimes W_{a_2}(r_2) \otimes ... \otimes W_{a_n} (r_n)
\]
\end{mycon}

\begin{proof}

By the lemmas and proposition we proved in this subsection, an irreducible type $\mathbf{1}$ comodule of $\widehat{SL_q(2)}$ will correspond to an irreducible type $\bf{1}$ module of $U_q(\widehat{\mathfrak{sl}_2})$. The latter have been classified by Char and Pressley, see Theorem $\ref{chari}$. They are isomorphic to tensor products of the form $V_{a_1}(r_1) \otimes ... \otimes V_{a_n}(r_n)$. We already know there are comodules $W_{a_n}(r_n) \otimes ... \otimes W_{a_1}(r_1)$ that correspond to them, so by the uniqueness result in Lemma $\ref{lemma:uniqueness}$ the proof is complete.   

\end{proof}

This result tells us that the irreducible type $\mathbf{1}$ modules of $U_q(\widehat{\mathfrak{sl}_2})$ are the same as the irreducible comodules of $\widehat{SL_q(2)}$. To actually prove this result, one can try to show the duality relation is non-degenerate. This would have other implications as well. It is known that the full category of finite dimensional modules of $U_q(\widehat{\mathfrak{sl}_2})$ is not semisimple (for the case $q \to 1$ see $\cite{ChariMoura}$) and not very well understood. A non-degenerate duality relation might allow us to study the modules of $U_q(\widehat{\mathfrak{sl}_2})$ by looking at them as comodules of $SL_q(2)$; in the same vein as using both the standard and the Drinfel'd presentation of $U_q(\widehat{\mathfrak{sl}_2})$ to study its finite dimensional representations.

A different approach to categorizing all irreducibles would be to simple develop the theory of highest weight comodules for $\widehat{SL_q(2)}$, similar to how it is done in $\cite{ParshallWang}$ for $SL_q(n)$.


\section{The free fermionic bialgebra}

\subsection{A parametrized YBE with noncommutative group}

We will now exhibit a parametrized YBE with non-abelian parameter group as given in $\cite{Korepin}$ or $\cite{Bump2}$.  

Let $\Gamma$ be the subgroup of $GL(4)$ with elements 
\begin{equation*}
x = \left ( \begin{array}{cccc}
c_1(x) & 0 & 0 & 0 \\
0 & a_1(x) & b_2(x) & 0 \\
0 & -b_1(x) & a_2(x) & 0 \\
0 & 0 & 0 & c_2(x) \end{array} \right )
\end{equation*} 
such that 
\begin{equation} \label{propertyparametrizedYBE}
a_1(x) a_2(x) + b_1(x) b_2(x) = c_1(x) c_2(x)
\end{equation}

Note that $\Gamma \simeq $GL$(2, \mathbb{C}) \times $GL$(1, \mathbb{C})$. The multiplication on $\Gamma$ is as follows for $x, y \in \Gamma$, $z=x \circ y$:
\begin{equation}
\begin{split}
a_1(z) &= a_1(x) a_1(y) - b_2(x) b_1(y) \\
a_2(z) &= a_2(x) a_2(y) - b_1(x)b_2(y) \\
b_1(z) &= b_1(x)a_1(y) + a_2(x) b_1(y)\\
b_2(z) &= a_1(x) b_2(y) + b_2(x) a_2(y)\\
c_1(z) &= c_1(x) c_1(y)\\
c_2(z) &= c_2(x) c_2(y)
\end{split}
\end{equation}

Let $V$ be a two dimensional vector space with a fixed basis $\{ v_1, v_2\}$. We define $R(x) \in \text{End}(V \otimes V)$ by the following formula:
\begin{equation*}
R(x) = \left ( \begin{array}{cccc}
a_1(x) & 0 & 0 & 0 \\
0 & b_1(x) & c_1(x) & 0 \\
0 & c_2(x) & b_2(x) & 0 \\
0 & 0 & 0 & a_2(x) \end{array} \right )
\end{equation*}

The following was noticed by Korepin $\cite{Korepin}$ and later rediscovered in $\cite{Bump2}$. 

\begin{mythm}
$R(x)$ is a solution to the parametrized YBE with parameter group $\Gamma \simeq GL(2, \mathbb{C}) \times GL(1, \mathbb{C})$. Namely, for all $x, y \in \Gamma$ the following equation holds:
\begin{equation*}
R_{12}(x) R_{13}(x \circ y) R_{23}(y) = R_{23}(y) R_{13}(x \circ y) R_{12}  (x)
\end{equation*}
\end{mythm}

\subsection{Motivation}

There are many reasons to study such a Hopf algebra. We focus on two in this section. 

First of all note that the matrix $R_q(x)$ associated to $U_q(\widehat{\mathfrak{sl}_2})$ defined in $\ref{eq:affine_YBE_solution}$ is free fermionic when $q=\pm i$. This means it will satisfy the property in equation $\ref{propertyparametrizedYBE}$. That is because 
\[
(q-xq^{-1})^2 +(1-x)^2 = x(q-q^{-1})^2
\]
when $q=\pm i$. This means that the Hopf algebra we will build $\mathcal{A}_{ff}$ will be an expansion of $\widehat{SL_{\pm i}(2)}$.

Another interesting fact is that one can look at the solution of the graded parametrized YBE corresponding to the quantum group $U_q(\widehat{\mathfrak{gl}(1|1)})$ $\cite{Zhang1}$. By multiplying certain entries with a minus sign as explained in $\cite{Kojima}$ one gets an ungraded solution of the parametrized YBE (what we called parametrized YBE so far is the same as ungraded parametrized YBE). This is just the Perk-Schultz solution $R^{PS}_q(x)$ given by $\cite{PerkSchultz}$: 
\[
R^{PS}_q(x) =\left ( \begin{array}{cccc}
q-xq^{-1} & 0 & 0 & 0 \\
0 & 1-x & x(q-q^{-1}) & 0 \\
0 & q-q^{-1} & 1-x & 0 \\
0 & 0 & 0 & -q^{-1}+xq \end{array} \right )
\]  

Note that $R^{PS}_q(x)$ is free fermionic for any $q$ because:
\[
(q-xq^{-1})(-q^{-1}+xq) + (1-x)^2 = x(q-q^{-1})^2
\]

This means the representation theory of $\mathcal{A}_{ff}$ is related not only to the representation theory of $U_{\pm i}(\widehat{\mathfrak{sl}_2})$, but also to that of $U_q(\widehat{\mathfrak{gl}(1|1)})$ for any $q$.

The second reason why this object is worth studying is because of its relation to Whittaker functions on $p$-adic groups. It was shown $\cite{Bump1}$ $\cite{Bump2}$ that certain values of spherical Whittaker functions on GL$(r, F)$, where $F$ is a nonarchimedean local field can be written down as the partition function of a six-vertex model in the spirit of Baxter.  

The weights of such a model form an $R$-matrix $R_{\Gamma}(z)$
\[
R^{\Gamma}(z) =\left ( \begin{array}{cccc}
1 & 0 & 0 & 0 \\
0 & t & (1+t)z & 0 \\
0 & 1 & z & 0 \\
0 & 0 & 0 & z \end{array} \right )
\]  
that also satisfies the free fermionic condition. It was then shown that there is a matrix $R^{\Gamma \Gamma} (z)$ that makes possible a YBE for $R^{\Gamma}(z)$:
\[
R^{\Gamma \Gamma}_{12} (x) R^{\Gamma}_{13}(xy) R^{\Gamma}_{23}(y) = R^{\Gamma}_{23} (y) R^{\Gamma}_{13}(xy) R^{\Gamma \Gamma}_{12}(x) 
\] 

$R^{\Gamma \Gamma} (z)$ will also satisfy the property in equation $\ref{propertyparametrizedYBE}$. The interesting thing is $R^{\Gamma \Gamma} (z)$ is the $R$-matrix corresponding to the standard representation of a Drinfeld twist of $U_q(\widehat{\mathfrak{gl}(1|1)})$ $\cite{Buciumas}$. 

By understanding the Hopf algebra $\mathcal{A}_{ff}$, we can then interpret the horizontal and vertical edges in the partition function mentioned above as comodules of $\mathcal{A}_{ff}$. It would be interesting to see what the representation theory of $\mathcal{A}_{ff}$ can tell us about Whittaker functions on $p$-adic groups.

\subsection{Construction}

Let $x, y \in \Gamma$. The bialgebra $A_{ff}$ is obtained by applying the reconstruction theorem to the braided monoidal category generated by $V_x$ for $x \in \Gamma$ with braiding on the generators given by $\tau R(y x^{-1}) : V_x \otimes V_y \to V_y \otimes V_x$. $A_{ff}$ will be generated by $t_{11}(x), t_{12}(x), t_{21}(x), t_{22}(x)$ for $x \in \Gamma$ subject to the relation $R(yx^{-1}) T_1(x) T_2(y) = T_2(y) T_1(x) R(yx^{-1})$ which can be expanded as follows:

\small 
\begin{equation*}
\begin{split}
\left ( \begin{array}{cccc}
a_1(yx^{-1}) &  &  &  \\
 &  b_1(yx^{-1}) & c_1(yx^{-1}) &  \\
 &  c_2(yx^{-1}) & b_2(yx^{-1}) &  \\
 &  &  &  a_2(yx^{-1})  \end{array} \right )
\left ( \begin{array}{cccc}
t_{11}(x)t_{11}(y) & t_{21}(x)t_{11}(y) & t_{11}(x)t_{21}(y) & t_{21}(x)t_{21}(y) \\
t_{12}(x)t_{11}(y) & t_{22}(x)t_{11}(y) & t_{12}(x)t_{21}(y) & t_{22}(x)t_{21}(y) \\
t_{11}(x)t_{12}(y) & t_{21}(x)t_{12}(y) & t_{11}(x)t_{22}(y) & t_{21}(x)t_{22}(y) \\
t_{12}(x)t_{12}(y) & t_{22}(x)t_{12}(y) & t_{12}(x)t_{22}(y) &  t_{22}(x)t_{22}(y) \end{array} \right ) =  \\
\left ( \begin{array}{cccc}
t_{11}(y)t_{11}(x) & t_{11}(y)t_{21}(x) & t_{21}(y)t_{11}(x) & t_{21}(y)t_{21}(x) \\
t_{11}(y)t_{12}(x) & t_{11}(y)t_{22}(x) & t_{21}(y)t_{12}(x) & t_{21}(y)t_{22}(x) \\
t_{12}(y)t_{11}(x) & t_{12}(y)t_{21}(x) & t_{22}(y)t_{11}(x) & t_{22}(y)t_{21}(x) \\
t_{12}(y)t_{12}(x) & t_{12}(y)t_{22}(x) & t_{22}(y)t_{12}(x) &  t_{22}(y)t_{22}(x) \end{array} \right )  
\left ( \begin{array}{cccc}
a_1(yx^{-1}) &  &  &  \\
 &  b_1(yx^{-1}) & c_1(yx^{-1}) &  \\
 &  c_2(yx^{-1}) & b_2(yx^{-1}) &  \\
 &  &  &  a_2(yx^{-1})  \end{array} \right )
\end{split}
\end{equation*}

\large

For each $x \in \Gamma$, $\mathcal{A}_{ff}$ will have $V_x$ as the standard two dimensional comodule with basis $\{ v_1, v_2\}$(we will not write the dependence of $v_1, v_2$ on $x$ as long as it can be deduced from context) and coaction 

\begin{equation*}
\begin{split}
v_1 \to t_{11}(x) \otimes  v_1 + t_{12}(x) \otimes v_2 \\
v_2 \to t_{21}(x) \otimes v_1 + t_{22}(x) \otimes v_2
\end{split}
\end{equation*}

The $RTT$ relation ensures the the following map is an $\mathcal{A}_{ff}$ comodule homomorphism between $V_x \otimes V_y$ and $V_y \otimes V_x$:
\begin{equation*}
\tau R(yx^{-1}) = \left ( \begin{array}{cccc}
a_1(yx^{-1}) &  &  &  \\
 &  c_2(yx^{-1}) & b_2(yx^{-1}) &  \\
 &  b_1(yx^{-1}) & c_1(yx^{-1}) &  \\
 &  &  &  a_2(yx^{-1})  \end{array} \right )
\end{equation*}

Let $\mathcal{B}^+ := \mathcal{A}_{ff} / \mathcal{I}^+$ and $\mathcal{B}^- := \mathcal{A}_{ff} / \mathcal{I}^-$ where $\mathcal{I}^+$ and $\mathcal{I}^-$ are the ideals generated by $t_{21}(x)$, $t_{12}(x)$ respectively. Let $\mathcal{T} := \mathcal{A}_{ff} / \mathcal{I} $ where $\mathcal{I}$ is the ideal generated by both $t_{12}(x)$ and $t_{21}(x)$ for all $x \in \Gamma$.    

The following relations will hold in $\mathcal{T}$:

\small
\begin{equation*}
\begin{split}
\left ( \begin{array}{cccc}
a_1(yx^{-1}) &  &  &  \\
 &  b_1(yx^{-1}) & c_1(yx^{-1}) &  \\
 &  c_2(yx^{-1}) & b_2(yx^{-1}) &  \\
 &  &  &  a_2(yx^{-1})  \end{array} \right )
\left ( \begin{array}{cccc}
t_{11}(x)t_{11}(y) &  &  &  \\
 & t_{22}(x)t_{11}(y) &  &  \\
 &  & t_{11}(x)t_{22}(y) &  \\
 &  &  &  t_{22}(x)t_{22}(y) \end{array} \right ) =  \\
\left ( \begin{array}{cccc}
t_{11}(y)t_{11}(x) &  &  &  \\
 & t_{11}(y)t_{22}(x) &  &  \\
&  & t_{22}(y)t_{11}(x) &  \\
&  &  &  t_{22}(y)t_{22}(x) \end{array} \right )  
\left ( \begin{array}{cccc}
a_1(yx^{-1}) &  &  &  \\
 &  b_1(yx^{-1}) & c_1(yx^{-1}) &  \\
 &  c_2(yx^{-1}) & b_2(yx^{-1}) &  \\
 &  &  &  a_2(yx^{-1})  \end{array} \right )
\end{split}
\end{equation*}

\large

\subsection{Representation theory of $\mathcal{A}_{ff}$}

In this section we characterize tensor products of standard $\mathcal{A}_{ff}$ comodules. Let $z = y x^{-1}$.

\begin{mythm}
$V_x \otimes V_y$ is irreducible if and only if $\tau R(z)$ is invertible.
\end{mythm}

\begin{proof}

$\tau R(z)$ has always rank at least one. It is a comodule map between $V_x \otimes V_y \to V_y \otimes V_x$ so if the map is not invertible it will have a kernel and therefore $V_x \otimes V_y$ will have a subcomodule. 

Note that the determinant of $\tau R(z)$ is $a_1(z)a_2(z) (c_1(z)c_2(z) - b_1(z) b_2(z)) = a_1^2(z) a_2^2(z)$. So $\tau R(z)$ is invertible if and only if $a_1(z) \neq 0$ and $a_2(z) \neq 0$. If that is the case we will show that $V_x \otimes V_y$ is irreducible.

The coaction is a map $\Delta : V_x \otimes V_y \to \mathcal{A}_{ff} \otimes V_x \otimes V_y$ which can be expanded to a coaction map $\Delta_{\mathcal{T}}  : V_x \otimes V_y \to \mathcal{T} \otimes V_x \otimes V_y$ that makes $V_x \otimes V_y$ into a $\mathcal{T}$ comodule. The formula for $\Delta_{\mathcal{T}}$ is given below:

\begin{equation*}
\begin{split}
\Delta_{\mathcal{T}} (v_1 \otimes v_1) =& t_{11}(x) t_{11}(y) \otimes v_1 \otimes v_1 \\
\Delta_{\mathcal{T}} (v_1 \otimes v_2) =& t_{11}(x) t_{22}(y) \otimes v_1 \otimes v_2  \\
\Delta_{\mathcal{T}} (v_2 \otimes v_1) =& t_{22}(x) t_{11}(y) \otimes v_2 \otimes v_1  \\
\Delta_{\mathcal{T}} (v_2 \otimes v_2) =& t_{22}(x) t_{22}(y) \otimes v_2 \otimes v_2 
\end{split}
\end{equation*}

Any nontrivial $\mathcal{A}_{ff}$ subcomodule $U \in V_x \otimes V_y$ will also be a $\mathcal{T}$ subcomodule.

Notice that as a $\mathcal{T}$ comodule $V_x \otimes V_y$ splits as a direct sum of 4 one dimensional subspaces $\text{Span}(v_1 \otimes v_1) \oplus \text{Span}(v_1 \otimes v_2) \oplus \text{Span}(v_2 \otimes v_1) \oplus \text{Span}(v_2 \otimes v_2)$. The elements $t_{11}(x) t_{11}(y), t_{11}(x) t_{22}(y),$ $ t_{22}(x) t_{11}(y), t_{22}(x) t_{22}(y)$ are linearly independent in $T_{ff}$ due to the fact that $\tau R(z)$ is invertible, so $W$ has to be a direct sum of one or several of the 4 subspaces. This means that $U$ contains at least one of the elements $v_1 \otimes v_1, ..., v_2 \otimes v_2$. If it contains at least one, the coaction on $\mathcal{A}_{ff}$ will force it to contain all four elements because $t_{ij}(x)t_{kl}(y) \neq 0$ which is due to the fact that $\tau R(z)$ is invertible. Therefore $U = V_x \otimes V_y$.

\end{proof}

We now classify the submodules and quotient modules of $V_x \otimes V_y$:

Case 1: $a_1(z) = a_2(z) = 0$.

$\tau R(z)$ has a three dimensional kernel with basis $\{ v_1 \otimes v_1,  c_1(z) v_1 \otimes v_2 - b_1(z) v_2 \otimes v_1, v_2 \otimes v_2  \}$. $\tau R(z^{-1}): V_y \otimes V_x \to V_x \otimes V_y$ has a one dimensional image with basis $\{ c_1(z) v_1 \otimes v_2 - b_1(z) v_2 \otimes v_1 \}$. Therefore $V_x \otimes V_y$ has a irreducible subcomodule with basis $\{ c_1(z) v_1 \otimes v_2 - b_1(z) v_2 \otimes v_1 \}$. It turns out it will also have a two dimensional irreducible quotient submodule with basis $\{ v_1 \otimes v_1, v_2 \otimes v_2 \}$. The three dimensional kernel of $\tau R(z)$ will split as a direct sum of the irreducible one dimensional and irreducible two dimensional. These are the only subcomodules of $V_x \otimes V_y$. The coaction will act on the one dimensional subcomodule as follows: 
\begin{equation*}
v \to \left(t_{11}(x) t_{22}(y) + \frac{b_2(z)}{c_2(z)} t_{21}(x) t_{12}(y) \right) \otimes v
\end{equation*} 

Denote the two dimensional comodule $W_{x,y}$. The coaction will act on it as follows: 
\begin{equation*}
\begin{split}
v_1 \otimes v_1 \to t_{11}(x) t_{11}(y) \otimes v_1 \otimes v_1 +   t_{12}(x) t_{12}(y) v_2 \otimes v_2 \\
v_2 \otimes v_2 \to t_{21}(x) t_{21}(y) \otimes v_1 \otimes v_1 +   t_{22}(x) t_{22}(y) v_2 \otimes v_2
\end{split}
\end{equation*}

Case 2: $a_1(z) = 0,  a_2(z) \neq 0$.

In this case $Ker(\tau R(z)) = Im( \tau R(z^{-1}))$, so $V_x \otimes V_y$ will have only one irreducible subcomodule of dimension two with basis $\{ v_1 \otimes v_1, c_1(z) v_1 \otimes v_2 - b_1(z) v_2 \otimes v_1 \}$ and one irreducible quotient comodule also of dimension two.

Case 3: $a_1(z) \neq 0,  a_2(z) = 0$.

Similar to Case 2, but now the subcomodule will have basis $\{ v_2 \otimes v_2, c_1(z) v_1 \otimes v_2 - b_1(z) v_2 \otimes v_1 \}$.

Note that for every $x,y \in \Gamma$ such that $a_1(z) = a_2(z ) = 0$, we have discovered a new two dimensional irreducible comodule $U_{x,y}$ that is not isomorphic to any of the standard comodules $V_w$. The braiding between $U_{x,y} \otimes V_w$ and $V_w \otimes U_{x,y}$ will be given by 
\small
\begin{equation*}
\left ( \begin{array}{cccc}
a_1(xw^{-1}) a_1(yw^{-1}) &  &  &  \\
 &   & b_2(xw^{-1}) b_2(yw^{-1}) &  \\
 & b_1(xw^{-1}) b_1(yw^{-1})  &  &  \\
 &  &  &  a_2(xw^{-1}) a_2(yw^{-1})  \end{array} \right )
\end{equation*}
\large
\vspace{1cm}

\subsection{An irreducibility criterion}

We now prove two lemmas that will help us in deciding which tensor products of standard comodules are irreducible. 

\begin{lemma} \label{lemma:linearindependence}
The set consisting of all $t_{i_1 i_1}(x_1) t_{i_2 i_2} (x_2) ... t_{i_n i_n}(x_n)$ for $i_k \in \{ 1, 2 \}$ is linearly independent in $\mathcal{T}$ if $\tau R(x_i x^{-1}_j)$ is invertible for all $j \leq i \in \{1, .., n \}$.
\end{lemma}

\begin{proof}
Before dividing by the $RTT$ ideal, the set of elements of the type $\,\,\,$ $t_{i_1 i_1} (x_{\sigma (1)}) t_{i_2 i_2} (x_{\sigma (2)}) ... t_{i_n i_n}(x_{\sigma(n)})$ for all $\sigma \in S_n$ are linearly independent. Once we divide, we will have relations between elements of the type $ t_{i_1 i_1}(x_{\sigma (1)}) ... t_{i_n i_n}(x_{\sigma(n)}) $ for fixed $\sigma \in S_n$ and elements of the type $t_{i_1 i_1}(x_{\rho (1)}) t_{i_2 i_2} (x_{\rho (2)}) ... t_{i_n i_n}(x_{\rho(n)}) $ for fixed $\rho \in S_n$. The functions that map elements of the first type to elements of the second type will consist of iterations of $\tau R(x_i x_j^{-1})$ tensored with the identity, therefore it will be invertible and unique, so no new relations will actually be forced between elements of the type $t_{i_1 i_1}(x_{\sigma (1)}) t_{i_2 i_2} (x_{\sigma (2)}) ... t_{i_n i_n}(x_{\sigma(n)})$. They will thus remain linearly independent. 

\end{proof}

\begin{lemma} \label{lemma:linearindependence2}
The elements $t_{i_1 j_1} (x_1) t_{i_2 j_2}(x_2) ... t_{i_n j_n}(x_n)$ are linearly independent in $\mathcal{A}_{ff}$ if $\tau R(x_j x^{-1}_i)$ is invertible for all $j \leq i \in \{1, .., n \}$.
\end{lemma}

\begin{proof}
Based on the same idea as the previous lemma, one can show that all elements of the type $t_{i_1 j_1} (x_1) t_{i_2 j_2}(x_2) ... t_{i_n j_n}(x_n)$ are linearly independent in $\mathcal{A}_{ff}$ when all $\tau R(x_j x^{-1}_i)$ are invertible.
\end{proof}

\begin{mythm}
$V_{x_1} \otimes V_{x_2} \otimes ... \otimes V_{x_n}$ is irreducible if and only if $\tau R(x_j x^{-1}_i)$ is invertible for all $j \leq i$.
\end{mythm}

\begin{proof}
If any of the maps $\tau R(x_j x_i^{-1})$ is not invertible, then there exists an $\mathcal{A}_{ff}$ comodule homomorphism between $... \otimes V_{x_j} \otimes ... \otimes V_{x_i} \otimes ... \to ... \otimes V_{x_i} \otimes ... \otimes V_{x_j} \otimes ...$ that will have a nontrivial kernel, therefore $V_{x_1} \otimes V_{x_2} \otimes ... \otimes V_{x_n}$ will not be irreducible.  

Suppose $W$ is a nontrivial subcomodule of $V_{x_1} \otimes V_{x_2} \otimes ... \otimes V_{x_n}$ and assume $\tau R(x_j x^{-1}_i)$ is invertible for all $j \leq i$.  $V_{x_1} \otimes .. \otimes V_{x_n}$ will split as a direct sum of one dimensional $T_{ff}$ comodules just like in the $n=2$ case. 

As a result of Lemma $\ref{lemma:linearindependence}$, all the coweights $t_{i_1 i_1}(x_1) t_{i_2 i_2} (x_2) ... t_{i_n i_n}(x_n)$ of the one dimensional comodules will be linearly independent in $\mathcal{T}$. Therefore $W$, as a $\mathcal{T}$-comodule, must split as a direct sum of some of the one dimensional subcomodules mentioned above. 

Because of this $W$ will be a direct sum of elements of the form $v_{i_1} \otimes ... \otimes v_{i_n}$. Assume $W$ is a proper subcomodule. There must be a $v_{j_1} \otimes ... \otimes v_{j_n} \notin W$ and a $v_{i_1} \otimes ... \otimes v_{i_n} \in W$. It follows that 
\[
\Delta (v_{i_1} \otimes ... \otimes v_{i_n}) = ... + t_{i_1,j_1} .. t_{i_n, j_n} v_{j_1} \otimes ... \otimes v_{j_n}+...
\]
where $t_{i_1,j_1} .. t_{i_n, j_n} v_{j_1}  \neq 0$. This gives us a contradiction. 

\end{proof}

Based on the dimension of all subcomodules of $V_x \otimes V_y$ and $V_x \otimes V_y \otimes V_z$, and also from the representation theory of finite dimensional modules of $U_{\pm i}(\widehat{\mathfrak{sl}_2})$ and $U_q(\widehat{\mathfrak{gl}(1|1)}$, we formulate the following conjecture:

\begin{mycon}
All finite dimensional comodules of $\mathcal{A}_{ff}$ will have dimension a power of two. 
\end{mycon}

We end this article with a few questions that might be suitable for further research and a thought on possible applications of this work.

Notice that $A_{ff}$ is not a Hopf algebra because it doesn't have an antipode, so it is natural to ask what relations to add in order to make $A_{ff}$ into a Hopf algebra. The quantum determinant  from the $\widehat{\mathfrak{sl}}_2$ case doesn't have a straightforward generalization to this case. One can set $\left(t_{11}(x) t_{22}(y) + \frac{b_2(z)}{c_2(z)} t_{21}(x) t_{12}(y) \right)$ equal to $1$ for when $a_1(z) = a_2(z)=0$. This would make a set of one dimensional comodules be isomorphic to the trivial comodule, but it would not make all of them. It would also not uniquely identify the antipode, since for a given $x$ there are many $y$'s such that $a_1(yx^{-1})=a_2(yx^{-1})=0$. 

$A_{ff}$ might have an interesting finite dimensional comodules, but what it does not have is infinite dimensional ones. Infinite dimensional representations of affine Lie algebras are very important, for example see the significance of the basic representation in theoretical physics $\cite{FrenkelKac}$. It would be interesting to build the ``dual'' of $\mathcal{A}_{ff}$, namely the object whose relation to $\mathcal{A}_{ff}$ is the same as the relation of $U_q(\widehat{\mathfrak{sl}_2})$ with $\widehat{SL_q(2)}$. One should then try to study the infinite dimensional representations of such an object.

\vspace{1cm}
\small
 
\bibliography{bib}
\bibliographystyle{amsalpha}

\vspace{1cm}

\large
\textsc{The Max Planck Institute for Mathematics, Bonn, 53111, Germany}

\textit{Email address:} \href{mailto:valentin.buciumas@gmail.com}{\nolinkurl{valentin.buciumas@gmail.com}}

\end{document}